\pdfoutput=1 %arXiv

\documentclass[12pt, letterpaper, leqno]{amsart}
\usepackage[letterpaper, margin=1in]{geometry}
\usepackage[utf8]{inputenc}
\usepackage[style=alphabetic, maxnames=99]{biblatex}
\AtBeginBibliography{\small}
\usepackage{microtype}
\microtypesetup{protrusion=false}
\usepackage{xpatch}
\usepackage[english]{babel}
\usepackage{amsfonts}
\usepackage{amsmath}
\usepackage{amssymb}
\usepackage{mathrsfs}
\usepackage{stmaryrd}
\usepackage{nicematrix}

\usepackage{amsthm}
\xpatchcmd{\proof}{\itshape}{\bfseries}{}{}
\usepackage{thmtools}
\usepackage{thm-restate}
\usepackage{tikz}
\usepackage{tikz-cd}
\usetikzlibrary{arrows.meta,automata,positioning}
\usetikzlibrary{svg.path}
\usetikzlibrary{decorations.pathreplacing}
\usepackage{amscd}
\usepackage{pdfpages}
\usepackage{float}
\usepackage[shortlabels]{enumitem}
\setlist[enumerate,1]{label={(\arabic*)}}
\usepackage[labelfont=bf]{caption}
\usepackage[linewidth=0pt]{mdframed}
\usepackage{subfig}
\usepackage{makecell}
\usepackage{multicol}
\usepackage{listings}
\usepackage{fancyhdr}
\usepackage[bottom, perpage]{footmisc}
\usepackage{etoolbox}

\patchcmd{\section}{\scshape}{\bfseries}{}{}
\makeatletter
\renewcommand{\@secnumfont}{\bfseries}
\makeatother

\allowdisplaybreaks

\usepackage{xr}
\usepackage{subfiles}

\usepackage{setspace}
\usepackage[shortcuts]{extdash}
\usepackage{xstring}
\usepackage{hyperref}
\usepackage[nameinlink,capitalize]{cleveref}
\usepackage{url}
\hypersetup{
	colorlinks   = true, 
	urlcolor     = cyan, 
	linkcolor    = blue, 
	citecolor   = blue,
	breaklinks = true
}

\numberwithin{equation}{section}
\numberwithin{figure}{section}

\makeatletter
\newcommand{\labeltext}[3][]{%
	\@bsphack%
	\csname phantomsection\endcsname% in case hyperref is used
	\def\tst{#1}%
	\def\labelmarkup{}% How to markup the label itself
	\def\refmarkup{}%
	\ifx\tst\empty\def\@currentlabel{\refmarkup{#2}}{\label{#3}}%
	\else\def\@currentlabel{\refmarkup{#1}}{\label{#3}}\fi%
	\@esphack%
	\labelmarkup{#2}% visible printed text.
}
\makeatother
\declaretheorem[numberwithin=section,numberlike=equation,style=plain]{theorem}

\declaretheorem[numbered=no,style=plain,name=Theorem]{theorem*}
\declaretheorem[numberwithin=section,numberlike=equation,style=plain]{proposition}

\declaretheorem[numbered=no,style=plain,name=Proposition]{proposition*}

\declaretheorem[numbered=no,style=plain,name=Lemma]{lemma*}
\declaretheorem[numberwithin=section,numberlike=equation,style=plain]{corollary}

\declaretheorem[numbered=no,style=plain,name=Corollary]{corollary*}

\declaretheorem[numbered=no,style=plain,name=Conjecture]{conjecture*}
\declaretheorem[numberwithin=section,numberlike=equation,style=plain]{question}

\declaretheorem[numbered=no,style=plain,name=Question]{question*}

\declaretheorem[numberwithin=section,numberlike=equation,style=definition]{definition}

\declaretheorem[numbered=no,style=definition,name=Definition]{definition*}
\declaretheorem[numberwithin=section,numberlike=equation,style=definition]{remark}

\declaretheorem[numbered=no,style=definition,name=Remark]{remark*}

\declaretheorem[numbered=no,style=definition,name=Notation]{notation*}

\declaretheorem[numbered=no,style=definition,name=Axiom]{axiom*}

\declaretheorem[numbered=no,style=definition,name=Construction]{construction*}

\declaretheorem[numbered=no,style=definition,name=Algorithm]{algorithm*}

\declaretheorem[numbered=no,style=definition,name=Property]{property*}

\declaretheorem[numberwithin=section,numberlike=equation,style=definition,qed=$\diamondsuit$]{example}

\declaretheorem[numbered=no,style=definition,name=Example,qed=$\diamondsuit$]{example*}
\usepackage{mystyle}

\newcommand{\define}[1]{\textcolor{magenta}{\emph{#1}}}

\title{Horospherical varieties with quotient singularities}
\author{Sean Monahan}
\address{Sean Monahan, Department of Mathematics, Technische Universit\"at M\"unchen, Boltzmannstra{\ss}e 3, 85748 Garching bei M\"unchen, Germany}
\email{sean.monahan@cit.tum.de}

\subjclass[2020]{14M27, 14L30, 14B05, 05E14}
\keywords{Horospherical, quotient singularity, local-to-global, Cox construction}

\begin{document}
	
\begin{abstract}
	Our main result is a combinatorial characterization of when a horospherical variety has (at worst) quotient singularities. Using this characterization, we show that every quasiprojective horospherical variety with quotient singularities is \textit{globally} the quotient of a smooth variety by a finite abelian group.
\end{abstract}	
	
\maketitle
%\setcounter{tocdepth}{1}
%\tableofcontents

\section{Introduction}\label{sec:introduction}

Quotient singularities are ubiquitous in all areas of geometry because they form a reasonable class of singularities to first study when expanding into the non-smooth realm. Recall that a variety $X$ over an algebraically closed field $k$ of characteristic $0$ is said to have (at worst) \define{quotient singularities} if it is \'{e}tale locally the quotient of a smooth variety by a finite group.\footnote{We always implicitly mean ``at worst", i.e. smooth varieties have quotient singularities in our terminology.} It is an easy fact that any variety which is globally of the form $V/F$ where $V$ is a smooth variety and $F$ is a finite group has quotient singularities. This paper is motivated by a question, originally posed by William Fulton and stated in \cite[Question 1.1]{satriano2015quotients}, which asks about the converse of this fact.

\begin{question}[Fulton's question]\label{fulton's question}
	If $X$ is a variety with quotient singularities, does there exist a smooth variety $V$ and a finite group $F$ acting on $V$ such that $X=V/F$?
\end{question}

\cref{fulton's question} was originally studied in \cite{satriano2015quotients},\footnote{In \cite{satriano2015quotients}, the authors allow for $k$ to be non-algebraically closed or to have positive characteristic. Note that, in positive characteristic, one should study \textit{tame} quotient singularities and be more careful regarding the characteristic of $k$.} where the authors show that there is an affirmative answer for quasiprojective toric varieties. In fact, they show that there is an affirmative answer for any quasiprojective variety which is globally the quotient of a smooth variety by a diagonalizable group. This applies to toric varieties via the Cox construction \cite{cox1995homogeneous}, which is a way of writing a toric variety with quotient singularities as the quotient of a smooth toric variety by a subgroup of its torus. 

Following the foundational work of Hu and Keel on Mori dream spaces \cite{hu2000mori}, the Cox construction has been extended to a broad range of varieties. In particular, this construction is well understood for the class of horospherical varieties, going back to the work of Brion \cite{brion2007total-coordinate}; we give a review of this construction in \cref{subsec:cox construction}. Recall that horospherical varieties are normal varieties equipped with an action from a reductive algebraic group $G$ for which there is an open orbit whose points are stabilized by maximal unipotent subgroups (see \cref{subsec:characterization of quotient singularities} below for more); these objects generalize toric varieties from the case where $G=T$ is a torus. 

%An important difference from the toric setting is that, if $X$ is a horospherical variety, then the variety arising from the Cox construction is only factorial in general, and not necessarily smooth. 
Much is known about the singularities of horospherical varieties (e.g. factorial, canonical, and terminal) \cite{pasquier2017survey-singularities}, but not about quotient singularities specifically. In particular, it was previously unknown if the variety arising from the Cox construction is smooth when $X$ has quotient singularities. 
This is the motivation for our main result (\cref{characterization of quotient singularities}), which is a combinatorial characterization of when a horospherical variety has quotient singularities. Equipped with this characterization, we are able to deduce that the variety arising in the Cox construction is indeed smooth when $X$ is horospherical with quotient singularities. Therefore, we may apply the result of \cite{satriano2015quotients} to answer \cref{fulton's question} in the affirmative for quasiprojective horospherical varieties. 

\begin{theorem}[Fulton's question for horospherical varieties]\label{fulton's question for horospherical varieties}
	\cref{fulton's question} has an affirmative answer for any quasiprojective horospherical variety $X$ with quotient singularities. In fact, $F$ may be taken to be a finite diagonalizable (hence abelian) group. 
\end{theorem}

\begin{remark}\label{fulton's global quotient need not be horospherical}
	In \cref{fulton's question for horospherical varieties}, we have that $X=V/F$, but we are \textit{not} saying that $V$ is a horospherical variety or that the quotient map $V\to X$ is a horospherical morphism. For example, \cite[Proposition 4.1]{satriano2015quotients} shows that the weighted projective space $\P(1,1,2)$ blown up at a torus-invariant point, which is a toric variety (hence horospherical), cannot be written as the quotient of a smooth \textit{toric} variety by a finite subgroup of its torus. 
\end{remark}

\subsection{Characterization of quotient singularities}\label{subsec:characterization of quotient singularities}

To state our combinatorial characterization of quotient singularities, we first establish some notation (following \cite{monahan2023overview}). For an introduction to horospherical (and more generally, spherical) varieties, we recommend \cite{knop1991luna,perrin2018sanya,monahan2023overview} or \cite[Chapter 5]{timashev2011homogeneous} for a more extensive resource.

Let $G$ be a connected reductive (linear) algebraic group over $k$ (e.g. $\SL_n$ or a torus $\G_m^n$), and let $S$ be the set of simple roots for $G$ with respect to a fixed maximal torus $T\subseteq G$ (e.g. diagonal matrices) and Borel subgroup $B\subseteq G$ (e.g. upper triangular matrices). Relative to the opposite Borel $B^-$ (e.g. lower triangular matrices), let $H\subseteq G$ be a horospherical subgroup, i.e. $H$ is closed and contains the unipotent radical of $B^-$. Let $P:=N_G(H)$, which is a parabolic subgroup of $G$ containing $B^-$; it is known that $P$ corresponds uniquely to a subset $I\subseteq S$ of simple roots for a Levi subgroup of $P$. 

A \define{$G/H$-horospherical variety} is a normal $G$-variety $X$ where there is an open orbit $G\cdot x_0\subseteq X$ such that the stabilizer of the base point $x_0$ is $H$; we can view this open orbit as the homogeneous space $G/H$. When $G=T$ is a torus we recover toric varieties, and when $H=P$ is a parabolic subgroup we recover the flag variety $G/P$. 

There is a combinatorial theory for horospherical varieties which goes back to the work of Vinberg and Popov \cite{vinberg1972class-quasihomogeneous} on affine horospherical varieties, and was extended by Luna and Vust \cite{luna1983plongements} to embeddings of homogeneous spaces including spherical varieties (this is summarized in \cite{knop1991luna}). This theory tells us that $X$ corresponds to a so-called coloured fan $\Sigma^c$ on a lattice $N$. In the case where $G=T$ is a torus, this is precisely the correspondence between toric varieties and fans (there are no \textit{colours} in this case). 

The aforementioned lattice $N$ is the one-parameter subgroup lattice of the torus $P/H$; notice that $N=N(G/H)$ only depends on the homogeneous space $G/H$. This lattice is equipped with a map $\xi:\calC:=S\setminus I\to N$ from a \define{universal colour set} $\calC$ which associates a \define{colour} $\alpha\in\calC$ to a \define{colour point} $\xi(\alpha)=:u_\alpha\in N$; see \cite[6]{knop1991luna} (or \cite[Section 4]{monahan2023overview}) for details. 

A \define{coloured cone} $\sigma^c$ on this ``coloured lattice" $N$ is a pair $(\sigma,\calF)$ where $\sigma\subseteq N_\R$ is a strongly convex integral cone and $\calF\subseteq\calC$ satisfies $\xi(\calF)\subseteq\sigma\setminus\{0\}$. A \define{coloured fan} $\Sigma^c$ on $N$ is a finite collection of coloured cones on $N$ which satisfies two properties: (1) for every $\sigma^c=(\sigma,\calF)\in\Sigma^c$ and every face $\tau\subseteq\sigma$, the \define{coloured face} $\tau^c=(\tau,\calF\cap\xi^{-1}(\tau))$ is in $\Sigma^c$; and (2) for every pair $\sigma_i^c=(\sigma_i,\calF_i)\in\Sigma^c$ (for $i=1,2$), their intersection $\sigma_1^c\cap\sigma_2^c=(\sigma_1\cap\sigma_2,\calF_1\cap\calF_2)$ is a coloured face of each $\sigma_i^c$ (and is therefore in $\Sigma^c$). Let $\Sigma^c(1)$ (resp. $\sigma^c(1)$) denote the set of coloured rays in $\Sigma^c$ (resp. which are contained in $\sigma^c$), i.e. coloured cones $(\rho,\calF)$ where $\rho$ is one-dimensional, and let $u_\rho\in N$ denote the minimal generator of $\rho$; we call $(\rho,\calF)$ a \define{non-coloured ray} if $\calF=\varnothing$.  

The following properties of $\Sigma^c$, which extend the notion of a simplicial (resp. smooth) fan in toric geometry, have been well studied in its relation to the geometry of $X$.

\begin{definition}[Simplicial/regular]\label{simplicial/regular definition}
	Let $\Sigma^c$ be a coloured fan on $N$. For any $\sigma^c=(\sigma,\calF)\in\Sigma^c$, consider the multiset
	\begin{align}\label{eq:simplicial/regular multiset}
		\{u_\rho:(\rho,\varnothing)\in\sigma^c(1)\}\cup\xi(\calF)
	\end{align}
	consisting of minimal generators for the non-coloured rays in $\sigma^c$ and colour points coming from $\calF$ (note: we say \textit{multiset} so that we count colour points with multiplicity). We say that $\Sigma^c$ is \define{simplicial} (resp. \define{regular}) if, for each $\sigma^c\in\Sigma^c$, the multiset \cref{eq:simplicial/regular multiset} is $\R$-linearly independent (resp. is part of a $\Z$-basis for $N$).
\end{definition}

It is known that $X$ is $\Q$-factorial (resp. factorial) if and only if $\Sigma^c$ is simplicial (resp. regular); for instance, see \cite[Theorem 4.2.3]{perrin2018sanya}. This plays an important role in our characterization of quotient singularities, because varieties with quotient singularities are always $\Q$-factorial, so $\Sigma^c$ being simplicial is a necessary condition for $X$ to have quotient singularities. 

To extend this necessary condition to a characterization, we present the following heuristic. In general, being smooth is a stronger condition than being factorial, and having quotient singularities is stronger than being $\Q$-factorial. One can think of there as being a gap between smoothness and factoriality, and a gap between quotient singularities and $\Q$-factoriality. The hope is that these gaps are the same, i.e. there is some ``extra condition" such that being smooth is equivalent to being factorial plus having this ``extra condition", \textit{and} having quotient singularities is equivalent to being $\Q$-factorial plus having \textit{this same} ``extra condition". 

It is well-known that these gaps, in both cases, are non-existent for toric varieties. However, these gaps are certainly present for horospherical varieties in general (e.g. see \cref{example Q-factorial not quotient singularity}). There is a combinatorial characterization for smooth horospherical varieties, originally due to \cite[\nopp 3.5]{pauer1983glatte-einbettungen} in the special case when $H$ is a maximal unipotent subgroup, and extended to the general case in \cite[Theorem 2.6]{pasquier2006thesis} and \cite[Theorem 28.10]{timashev2011homogeneous} independently, which uses the following ``extra condition" -- regarding the colours in $\Sigma^c$ -- to fill in the gap between smoothness and factoriality.

\begin{restatable}[Vivid]{definition}{colourconditiondefinition}\label{vivid definition}
	Let $X$ be a $G/H$-horospherical variety with coloured fan $\Sigma^c$. Recall that $P=N_G(H)$ corresponds to the subset $I$ of simple roots of $G$. For $J\subseteq S$, let $\Gamma_J$ be the sub-Dynkin diagram of $G$ whose vertices are the simple roots in $J$. We say that $\Sigma^c$ (or $X$ itself) is \define{vivid} if, for each $\sigma^c=(\sigma,\calF)\in\Sigma^c$ and each $\alpha\in\calF$, the following conditions hold:
	\begin{enumerate}
		\item The connected component $\Gamma^\alpha$ of $\Gamma_{I\cup\calF}$ containing $\alpha$ contains no other element of $\calF$. 
		\item $\Gamma^\alpha=\Gamma_{\{\alpha\}\cup I_\alpha}$ for some $I_\alpha\subseteq I$, and $\Gamma^\alpha$ has type $\mathbf{A}_n$ or $\mathbf{C}_n$ (for some $n$) in which $\alpha$ is the first simple root (for $\mathbf{C}_n$ this means the opposite end to the longest simple root). 
	\end{enumerate} 
\end{restatable}

To be clear, \cite[Theorem 2.6]{pasquier2006thesis} and \cite[Theorem 28.10]{timashev2011homogeneous} say that $X$ is smooth if and only if $\Sigma^c$ is regular and vivid.\footnote{Instead of \textit{vivid}, the term \textit{lisse} (translating to \textit{smooth}) is used in \cite[Definition 2.4]{pasquier2006thesis}. We choose to avoid \textit{lisse} in this context because this alone does not correspond to smoothness of horospherical varieties, and we want to avoid confusion with the notion of a \textit{smooth fan} in polyhedral geometry.} Geometrically, one should think of vividness as saying that $X$ has a local model in terms of toric varieties; we elaborate on vividness and its geometric meaning in \cref{subsec:vividness}. 

The following theorem shows that this same ``extra condition" (vividness) fits perfectly into the gap between having quotient singularities and being $\Q$-factorial. 

\begin{restatable}[Characterization of quotient singularities]{theorem}{characterizationofquotientsingularities}\label{characterization of quotient singularities}
	Let $X$ be a $G/H$-horospherical variety with coloured fan $\Sigma^c$. Then $X$ has quotient singularities if and only if $\Sigma^c$ is simplicial and vivid. 
\end{restatable}

\begin{remark}\label{toroidal quotient singularities}
	In \cref{characterization of quotient singularities}, if $X$ is toroidal, meaning $\calF=\varnothing$ for each $\sigma^c=(\sigma,\calF)\in\Sigma^c$ (e.g. if $G=T$ is a torus and $X$ is a toric variety), then vividness holds vacuously, so $X$ has quotient singularities if and only if $\Sigma^c$ is simplicial (i.e. $X$ is $\Q$-factorial).
\end{remark}

\subsection{Open questions}

Knowing that horospherical varieties are part of the larger class of spherical varieties, which enjoy the combinatorial characterization via coloured fans (although the definition of \textit{coloured fan} is different from the one we are using), there are two natural questions that arise regarding extensions of our work. The first is on extending our main theorem to spherical varieties.

\begin{question}\label{open question-quotient singularities for spherical varieties}
	Is there a combinatorial characterization (on the coloured fan) for when a spherical variety has quotient singularities? 
\end{question}

Based on our heuristic, one might hope that the gap between smoothness and factoriality is the same as the gap between quotient singularities and $\Q$-factoriality for spherical varieties. For example, there is a known combinatorial characterization of smoothness for spherical varieties due to Gagliardi \cite[Theorem 1.2]{gagliardi2015smoothness-spherical},
%(and a ``half-combinatorial" criterion due to Brion \cite[Section 4.2]{brion1991geometrie-spheriques}) 
so we might hope that a spherical variety $X$ has quotient singularities if and only if its coloured fan is simplicial (so that $X$ is $\Q$-factorial) and conditions (2) and (3) of \cite[Theorem 1.2]{gagliardi2015smoothness-spherical} are satisfied. 

The second natural question is if Fulton's question holds for spherical varieties.

\begin{question}\label{open question-fulton's question for spherical varieties}
	What is the answer to \cref{fulton's question} for spherical varieties?
\end{question}

If one successfully addresses \cref{open question-quotient singularities for spherical varieties}, then they may be able to use our strategy for proving \cref{fulton's question for horospherical varieties} to give an answer to \cref{open question-fulton's question for spherical varieties}. Note that the Cox construction holds for spherical varieties, although our combinatorial description in \cref{subsec:cox construction} only applies to the horospherical setting.

\subsection{Fixed notation}

Throughout this paper we use the following notation. All varieties are assumed to be irreducible, and all varieties and algebraic groups are assumed to be defined over an algebraically closed field $k$ of characteristic $0$.

Regarding horospherical varieties and coloured fans, we use the notation and conventions from \cite{monahan2023overview}. Throughout the paper, we fix a connected reductive (linear) algebraic group $G$; we may assume that $G=G^{ss}\times T_0$ where $G^{ss}$ is semisimple simply connected and $T_0$ is a torus. Fix a Borel subgroup $B\subseteq G$ and a maximal torus $T\subseteq B$, and relative to these subgroups, let $S$ be the set of simple roots of $G$. Let $B^-$ be the opposite Borel subgroup, which is conjugate to $B$ in $G$ and satisfies $B\cap B^-=T$. Let $U\subseteq B^-$ denote the unipotent radical of $B^-$ (which is a maximal unipotent subgroup of $G$), and fix a horospherical subgroup $H\subseteq G$ containing $U$. Let $P:=N_G(H)\supseteq B^-$ denote the associated parabolic subgroup, which corresponds to $I\subseteq S$. 

Let $N=N(G/H)$ denote the coloured lattice associated to $G/H$ (equal to the one-parameter subgroup lattice of the torus $P/H$), which has universal colour set $\calC=\calC(G/H)$ and colour map $\xi:\calC\to N$ sending $\alpha\mapsto \xi(\alpha)=u_\alpha$. For a coloured cone $\sigma^c=(\sigma,\calF)$, we sometimes write $\calF(\sigma^c)$ to denote its colour set $\calF$. For a coloured fan $\Sigma^c$ on $N$, let $\Sigma^c(1)$ denote the set of coloured rays of $\Sigma^c$. Given a coloured ray $\rho^c=(\rho,\calF)\in\Sigma^c(1)$, let $u_\rho\in N$ denote the minimal generator for $\rho$. 

%For a $G/H$-horospherical variety $X$ and $\alpha\in\calC$, we let $D_\alpha^X$ denote the colour divisor in $X$ corresponding to $\alpha$, or we simply write $D_\alpha$ when there is no confusion. 
For a $G/H$-horospherical variety $X$, let $\calD(X)$ denote the set of $B$-invariant prime divisors in $X$.\footnote{
	The orbit $B\cdot x_0$ (which is contained in $G/H=G\cdot x_0$) is open in $X$ because we chose $H$ to contain the unipotent radical of the \textit{opposite} Borel $B^-$. Then the elements of $\calD(X)$ are the boundary components of this open $B$-orbit. Note that \cite{monahan2023overview} uses the opposite convention: the roles of $B$ and $B^-$ are swapped. 
}
Note that $\calD(X)$ is in natural bijection with the union of $\calC$ and the set of non-coloured rays in the coloured fan for $X$: the $D\in\calD(X)$ which are \textit{not} $G$-invariant correspond uniquely to colours in $\calC$, and the $D\in\calD(X)$ which are $G$-invariant correspond uniquely to non-coloured rays. 
Let $\calF(X)$ denote the union of all $\calF(\sigma^c)$ as $\sigma^c$ ranges over the coloured cones in the coloured fan for $X$. Note that the group $\Aut^G(X)$ of $G$-equivariant automorphisms of $X$ is isomorphic to the torus $P/H$, which has one-parameter subgroup lattice $N$; we denote this torus by $T_N$. 

To indicate the horospherical variety corresponding to $\Sigma^c$, we sometimes write $X_{\Sigma^c}$. If $\Sigma^c$ is generated by a single coloured cone (meaning it has one maximal coloured cone) $\sigma^c$, then we may abuse notation to write $\Sigma^c=\sigma^c$, and the corresponding variety $X_{\sigma^c}$ is called \define{simple} (equivalently, this means the variety has a unique closed $G$-orbit). 

Lastly, if $X_i$ are $G/H_i$-horospherical varieties for $i=1,2$, then a $G$-equivariant morphism $X_1\to X_2$ which restricts to the projection map $G/H_1\to G/H_2$ induced by an inclusion $H_1\subseteq H_2$ is called a \define{horospherical morphism}.

\section{Preliminaries}\label{sec:preliminaries}

We briefly review some aspects of horospherical varieties that are important for our proofs of the main results.

\subsection{Affine local structure}\label{subsec:affine local structure}

If $X=X_{\Sigma^c}$ is a $G/H$-horospherical variety, then it admits an open cover by the simple $G/H$-horospherical varieties $X_{\sigma^c}$ as $\sigma^c$ ranges over the coloured cones of $\Sigma^c$. Unlike in toric geometry, simple horospherical varieties are not necessarily affine; in fact, a horospherical variety is affine if and only if it is simple, so of the form $X_{\sigma^c}$ for a single coloured cone $\sigma^c$, and $\calF(\sigma^c)$ is the entire universal colour set $\calC$. 

Now consider the case when $\Sigma^c=\sigma^c$ is a single coloured cone, so $X$ is simple. We provide a quick summary of the affine local structure for $X$; compare with \cite[Theorem 28.2]{timashev2011homogeneous} (or \cite[Section 6.2]{monahan2023overview}). 

Let $Q\supseteq B^-$ be the parabolic subgroup of $G$ corresponding to $I\cup\calF(\sigma^c)\subseteq S$.
%; note that $Q^-$ is the stabilizer of the colour divisors coming from $\calC\setminus\calF(\sigma^c)$. 
Since $I\subseteq I\cup\calF(\sigma^c)$, we have $P\subseteq Q$, so we have a natural projection $G/H\to G/P\to G/Q$. The universal colour set $\calC(G/Q)$ for $G/Q$ is $\calC\setminus\calF(\sigma^c)$, so using facts about morphisms of horospherical varieties, this projection $G/H\to G/Q$ extends to a horospherical morphism $X\to G/Q$. Let $Z$ denote the fibre of the base point $eQ$; so $Z$ is a closed subvariety of $X$. Note that every fibre of $X\to G/Q$ is isomorphic to $Z$ since the map is $G$-equivariant and the base is a single $G$-orbit. In particular, we have $X=G\times^Q Z$; recall that the associated bundle $G\times^Q Z$ is $(G\times Z)/Q$ where $q\cdot (g,z):=(gq^{-1},q\cdot z)$. Since $eQ$ is a $Q$-fixed point in $G/Q$, we see that $Z$ is $Q$-invariant. Moreover, $\Rad_u(Q)$ acts on $Z$ trivially, so $Z$ is an $\calL$-variety where $\calL$ is the standard Levi subgroup of $Q$. Note that $\calL$ is a connected reductive algebraic group with simple root set $I\cup\calF(\sigma^c)$.  

The coloured lattice $N(\calL/(\calL\cap H))$ satisfies the following properties. As a lattice, it is the same as $N(G/H)$ because $\calL/(\calL\cap H)$ and $G/H$ have the same associated torus. However, the colour structure is different: the universal colour set is $\calC(\calL/(\calL\cap H))=\calF(\sigma^c)$. 

The punchline is that $Z$ is an affine $\calL/(\calL\cap H)$-horospherical variety whose associated coloured cone is $\sigma^c$ (the same as for $X$) viewed on the coloured lattice $N(\calL/(\calL\cap H))$. 
%When we refer to the ``affine local structure" of $X$, we mean this $Z$. 

\subsection{Cox construction}\label{subsec:cox construction}

In this subsection, we summarize the Cox construction for horospherical varieties. The Cox ring and Cox construction for spherical varieties was first studied by Brion \cite{brion2007total-coordinate}; our summary is a simplified version of Brion's treatment because we are only interested in horospherical varieties (compare with \cite[Section 5.2]{monahan2025horospherical-stack}). The general theory of this Cox construction is discussed in great detail in \cite{arzhantsev2015cox-rings} (primarily Sections 1.4 and 1.6, and see Section 4.5.4 on spherical varieties). 

Let $X$ be a $G/H$-horospherical variety with associated coloured fan $\Sigma^c$ on $N$. For this construction, we assume that $X$ has no torus factors, meaning that $X$ is \textit{not} $G$-equivariantly isomorphic to $X_1\times T_1$ where $X_1$ is a horospherical $G$-variety of smaller dimension than $X$ and $T_1$ is a nontrivial torus; see \cite[Section 5.6]{monahan2023overview} for full details on torus factors in horospherical varieties. Although we explain how to remove this assumption in \cref{torus factor in cox construction}. 

\subsubsection*{Step 1: Cox ring} The \define{Cox ring} of $X$ is 
\begin{align*}
	\Cox(X) := \bigoplus_{[D]\in\Cl(X)} \scrO_X(D)(X).
\end{align*}
When $X$ is toric, this is the same as the homogeneous coordinate ring constructed in \cite{cox1995homogeneous}, which is a polynomial ring. Since $X$ is horospherical, $\Cox(X)$ is a finitely generated $k$-algebra, so $\Spec(\Cox(X))$ is an affine variety. 

\subsubsection*{Step 2: $\wh G$ and $\wh K$ groups} This $\Spec(\Cox(X))$ has the structure of a horospherical $\wh G$-variety where $\wh G$ is a connected reductive algebraic group which can be determined as follows. 
Recall that $G=G^{ss}\times T_0$ where $G^{ss}$ is semisimple simply connected and $T_0$ is a torus. Since $\Cl(X)$ is finitely generated, $\wh K:=\Spec(k[\Cl(X)])$ is a diagonalizable algebraic group, and the identity component $\wh K^\circ$ is a torus of rank $\rank(\Cl(X))=\#\calD(X)-\rank(N)$. Then we may take $\wh G=G^{ss}\times (T_0\times \wh K^\circ)=G\times \wh K^\circ$. Moreover, $\Spec(\Cox(X))$ is a $\wh G/\wh H$-horospherical variety for some $\wh H\subseteq\wh G$. 

More specifically, we can define $\wh H$ as follows. If $n=\#\calD(X)$, then the surjection of character lattices $\Z\calD(X)\cong \Z^n\to \Cl(X)$ (since $\calD(X)$ generates $\Cl(X)$) yields an injection of diagonalizable groups $\wh K\subseteq \G_m^n$. Using the morphism $G\to X$ via $g\mapsto g\cdot x_0$ (where $x_0$ is in the open $G$-orbit), we can pull back each $D\in\calD(X)$ to a single equation $f_D\in k[G]$ (chosen uniquely by requiring $f_D$ to be $T_0$-invariant and $f_D(1)=1$; compare with \cite[Section 4.1]{brion2007total-coordinate}), and $H$ acts on $f_D$ by a character $\chi_D$; note that if $D$ is $G$-invariant, then $f_D=1$ and $\chi_D=1$. If $\calD(X)=\{D_1,\ldots,D_n\}$, then we can define 
\begin{align*}
	\wh H := \{(h,\chi_{D_1}(h),\ldots,\chi_{D_n}(h)) : h\in H\} \subseteq \wh G \subseteq G\times\G_m^n.
\end{align*}

\subsubsection*{Step 3: $\wh N$ coloured lattice} Let $\wh N$ be the coloured lattice associated to $\wh G/\wh H$. Then $\wh N$ has rank $n:=\#\calD(X)$, which is equal to the sum of $\#\calC$ and the number of non-coloured rays in $\Sigma^c$. Furthermore, $\wh N$ has universal colour set $\calC$ (the same as $N$), and the colour points $\wh u_\alpha\in\wh N$ are part of a $\Z$-basis. Thus, we can decompose $\wh N$ into the following standard basis:
\begin{align}\label{eq:cox lattice basis}
	\wh N = \left(\bigoplus_{\alpha\in\calC} \Z \cdot e_\alpha\right) \oplus \left(\bigoplus_{(\rho,\varnothing)\in\Sigma^c(1)} \Z \cdot e_\rho\right)
\end{align}
where each $e_\alpha$ is a basis vector equal to the colour point $\wh u_{\alpha}$, and each $e_\rho$ is a basis vector corresponding to the non-coloured ray $(\rho,\varnothing)$ in $\Sigma^c$. 

Then the coloured fan associated to $\Spec(\Cox(X))$ is the single coloured cone generated by all basis vectors $(\Cone(e_\alpha,~ e_\rho : \alpha\in\calC,~ (\rho,\varnothing)\in\Sigma^c(1)),~ \calC)$, i.e. the first orthant in $\R^n$. 

\subsubsection*{Step 4: $\mu$ map} The projection map $\wh G/\wh H\to G/H$ induced by $\wh G=G\times \wh K^\circ\to G$ induces a map of coloured lattices $\mu:\wh N\to N$, which explicitly is given by
\begin{align}\label{eq:cox lattice map}
	\mu(e_\alpha)=u_\alpha ~~\forall \alpha\in\calC, \qquad \mu(e_\rho)=u_\rho ~~\forall (\rho,\varnothing)\in\Sigma^c(1).
\end{align}

Note that the $\R$-linear extension $\mu_\R:\wh N_\R\to N_\R$ is surjective because $X$ has no torus factors (see \cite[Proposition 5.39]{monahan2023overview}), so $\mu$ has finite cokernel. 

\subsubsection*{Step 5: Good quotient} The dual map of $\mu$ is part of the short exact sequence
\begin{align*}
	0 \map N^\vee \xrightarrow{\mu^\vee} \wh N^\vee\cong \bigoplus_{D\in\calD(X)} \Z D \map \Cl(X) \map 0
\end{align*}
and by applying $\Hom(-,\G_m)$, the map $\wh N^\vee\to \Cl(X)$ induces an inclusion of diagonalizable groups $\wh K\subseteq T_{\wh N}\cong \G_m^n$. Thus, $\wh K$ acts on $\Spec(\Cox(X))$ via $\wh G$-equivariant automorphisms. 

Now there exists a $\wh G$-invariant open subset $\wh X\subseteq \Spec(\Cox(X))$, which is itself a $\wh G/\wh H$-horospherical variety, which has the following two properties: $\codim(\Spec(\Cox(X))\setminus X)\geq 2$ and there is a $\wh G$-equivariant good quotient $\pi:\wh X\to X$ for the action of $\wh K$ (acting by $\wh G$-equivariant automorphisms). Note that $\pi$ is a horospherical morphism with associated map of coloured lattices $\mu:\wh N\to N$. 

The good quotient $\pi:\wh X\to X$ is geometric if and only if $\Sigma^c$ is simplicial (equivalently, $X$ is $\Q$-factorial), and furthermore $\wh K$ acts freely on $\wh X$ if and only if $\Sigma^c$ is regular (equivalently, $X$ is factorial).

\subsubsection*{Step 6: $\wh\Sigma^c$ coloured fan} Let $\wh\Sigma^c$ be the coloured fan on $\wh N$ associated to $\wh X$, which is a sub-coloured fan of $(\Cone(e_\alpha,~ e_\rho : \alpha\in\calC,~ (\rho,\varnothing)\in\Sigma^c(1)),~ \calC)$. We can describe $\wh\Sigma^c$ in terms of $\Sigma^c$ as follows. For each $\sigma^c=(\sigma,\calF)\in\Sigma^c$, let $\wh\sigma^c=(\wh\sigma,\wh\calF)$ be the coloured cone on $\wh N$ defined by
\begin{align*}
	\wh\sigma := \Cone(e_\alpha,~ e_\rho : \alpha\in\calF,~ \rho\subseteq\sigma) \quad\text{and}\quad \wh\calF:=\calF.
\end{align*}
Then $\wh\Sigma^c$ is the coloured fan on $\wh N$ which is generated by all such $\wh\sigma^c$. Note that $\wh\Sigma^c$ is regular, so $\wh X$ is factorial (but might not be smooth). 

\begin{remark}[Cox construction with torus factors]\label{torus factor in cox construction}
	If $X$ has a torus factor, then using \cite[Section 5.6]{monahan2023overview} we may write $X=X'\times T_{N/N'}$ where $N'$ is the (saturated) sublattice of $N$ spanned by all points in the fan $\Sigma$ together with all colour points in $N$, and $X'$ is the horospherical $G$-variety with coloured fan $\Sigma^c$ viewed on the coloured lattice $N'$. 
	
	Then performing the Cox construction on $X$ is essentially the same as performing it on $X'$. Explicitly, the Cox construction on $X'$ yields the good quotient $\wh{X'}\to X'$ for the action of $\wh{K'}$, and we can obtain the good quotient
	\begin{align*}
		\wh X:=\wh{X'}\times T_{N/N'} \map X'\times T_{N/N'}=X
	\end{align*}
	for the action of $\wh{K'}$, where $\wh{K'}$ acts trivially on $T_{N/N'}$. We can recover all the notation used above in the Cox construction for this situation with torus factors:
	\begin{itemize}[leftmargin=2.5em]
		\item $\wh G$ and $\wh K$ in the construction for $X$ are $\wh{G'}\times T_{N/N'}$ and $\wh{K'}\times\{1\}\cong \wh{K'}$, respectively. 
		\item $\wh N$ in the construction for $X$ is $\wh{N'}\times (N/N')$.
		\item $\wh\mu$ in the construction for $X$ is the product of $\wh{\mu'}:\wh{N'}\to N'$ with $\id:N/N'\to N/N'$.
		\item $\wh X$ in the construction for $X$ is $\wh{X'}\times T_{N/N'}$, as above.
		\item $\wh\Sigma^c$ for $\wh X$ is the same as $(\wh{\Sigma'})^c$ for $\wh{X'}$, but now considered on $\wh N$ rather than $\wh{N'}$. 
	\end{itemize}
\end{remark}

\subsection{Vividness}\label{subsec:vividness}

For convenience, we recall the definition of vivid, \cref{vivid definition}.
\colourconditiondefinition*

\begin{remark}\label{vivid Dynkin diagrams}
	If $\Gamma^\alpha$ is one of the connected components from \cref{vivid definition}, then $\Gamma^\alpha=\Gamma_{\{\alpha\}\cup I_\alpha}$ has one of the two forms below:
	\begin{equation*}
		\begin{tikzpicture}
			\def\a{1}
			\tikzset{dynkin/.style={circle,draw,inner sep=0pt,minimum size=2mm}}
			\path
			(0,0)    node{$\mathbf{A}_n$:}  
			++(0:\a)	node[dynkin] (N1) {} +(-90:.5) node{$\alpha$}
			++(0:\a)   node[dynkin] (N2) {} 
			++(0:\a) node[dynkin] (N3) {}
			++(0:0.5*\a) coordinate (A) ++(0:\a) coordinate (B)
			++(0:0.5*\a) node[dynkin] (N4) {} ;
			
			\draw[dashed] (A)--(B);
			\draw (N1)--(N2)--(N3)--(A) (B)--(N4);
			\draw[decorate,decoration={brace,raise=3mm},thick]
			(N4.center)--(N2.center) node[midway,below=4mm]{$I_\alpha$};
			
			\path
			(7*\a,0)      node{$\mathbf{C}_n$:}
			++(0:\a)	node[dynkin] (M1) {} +(-90:.5) node{$\alpha$}
			++(0:\a)   node[dynkin] (M2) {} 
			++(0:\a) node[dynkin] (M3) {}
			++(0:0.5*\a) coordinate (C) ++(0:\a) coordinate (D)
			++(0:0.5*\a) node[dynkin] (M4) {} 
			++(0:0.5*\a) coordinate (E)
			++(0:0.5*\a) node[dynkin] (M5) {} ;
			
			\draw[dashed] (C)--(D);
			\draw (M1)--(M2)--(M3)--(C) (D)--(M4);
			\draw[double,double distance=1mm] (M4)--(M5);
			\draw[-{Classical TikZ Rightarrow[length=2mm]},double,double distance=1mm] (M5)--(E.west);
			\draw[decorate,decoration={brace,raise=3mm},thick]
			(M5.center)--(M2.center) node[midway,below=4mm]{$I_\alpha$};
		\end{tikzpicture}
	\end{equation*}
	If $G^\alpha$ is the semisimple simply connected group associated to this Dynkin diagram $\Gamma^\alpha$, then $G^\alpha=\SL_{n+1}$ or $G^\alpha=\opn{Sp}_{2n}$, respectively. Having one of the two forms above is equivalent to $G^\alpha$ acting transitively on projective space (of dimension $n$ or $2n-1$, respectively) where the stabilizer (up to conjugation) is the parabolic subgroup of $G^\alpha$ corresponding to $I_\alpha$. 
\end{remark}

\begin{example}[{cf. \cite[Example 2.5]{pasquier2006thesis}}]\label{vivid examples}	
	Consider $G$ of type $\mathbf{C}_7$, so we may assume $G\cong \opn{Sp}_{14}\times T_0$ where $T_0$ is a torus.  
	\begin{equation*}
		\begin{tikzpicture}
			\def\a{1}
			\tikzset{dynkin/.style={circle,draw,inner sep=0pt,minimum size=2mm}}
			\path
			(0,0)   node[dynkin] (M1) {} +(-90:.5) node{$\alpha_1$}
			++(0:\a)   node[dynkin] (M2) {} +(-90:.5) node{$\alpha_2$}
			++(0:\a) node[dynkin] (M3) {} +(-90:.5) node{$\alpha_3$}
			++(0:\a) node[dynkin] (M4) {} +(-90:.5) node{$\alpha_4$}
			++(0:\a) node[dynkin] (M5) {} +(-90:.5) node{$\alpha_5$}
			++(0:\a) node[dynkin] (M6) {} +(-90:.5) node{$\alpha_6$}
			++(0:0.5*\a) coordinate (E)
			++(0:0.5*\a) node[dynkin] (M7) {} +(-90:.5) node{$\alpha_7$} ;
			
			\draw (M1)--(M2)--(M3)--(M4)--(M5)--(M6);
			\draw[double,double distance=1mm] (M6)--(M7);
			\draw[-{Classical TikZ Rightarrow[length=2mm]},double,double distance=1mm] (M7)--(E.west);
		\end{tikzpicture}
	\end{equation*}
	
	Let $\sigma^c=(\sigma,\calF)$ be a coloured cone on $N$ with $\calF\subseteq \calC=S\setminus I$. We give the following examples with $I$ and $\calF$:
	\begin{enumerate}[label=(\alph*)]
		\item If $I=\{\alpha_2,\alpha_6,\alpha_7\}$ and $\calF=\{\alpha_1,\alpha_5\}$, then $\sigma^c$ is vivid; note that $\Gamma^{\alpha_1}$ has type $\mathbf{A}_2$ and $\Gamma^{\alpha_5}$ has type $\mathbf{C}_3$. 
		\item If $I=\{\alpha_4,\alpha_6,\alpha_7\}$ and $\calF=\{\alpha_5\}$, then $\sigma^c$ is \textit{not} vivid because $\Gamma^{\alpha_5}$ has type $\mathbf{C}_4$ but $\alpha_5$ is not the first simple root. 
		\item If $I=\{\alpha_2,\alpha_5\}$ and $\calF=\{\alpha_1,\alpha_3\}$, then $\sigma^c$ is \textit{not} vivid because $\alpha_1$ and $\alpha_3$ are in the same connected component of $\Gamma_{I\cup\calF}$.  \qedhere
	\end{enumerate}
\end{example}

\begin{remark}[Vivid in Cox construction]\label{vivid in cox construction}
	Let $X$ be a $G/H$-horospherical variety, and let $\wh X$ be from the Cox construction of $X$. Since the coloured fan $\wh\Sigma^c$ for $\wh X$ is regular, it is clear that the following are equivalent:
	\begin{enumerate}
		\item $X$ is vivid.
		\item $\wh X$ is vivid.
		\item $\wh X$ is smooth. 
	\end{enumerate}
\end{remark}

\begin{proposition}[Vivid for affine varieties]\label{vivid for affine}
	Let $X$ be an affine $G/H$-horospherical variety with no torus factors. Then the following are equivalent:
	\begin{enumerate}
		\item $X$ is vivid.
		\item $G/P$ is a product of projective spaces.
		\item $X$ is a toric variety.\footnote{Here we mean that $X$ is intrinsically a toric variety, but we are not saying that $G$ has to be a torus. E.g. $X=\A^2$ is a horospherical $\SL_2$-variety, but it also admits the structure of a toric variety with torus $\G_m^2$.}
		\item $\wh X$ (from the Cox construction) is an affine space.
	\end{enumerate}
	
	\begin{proof}
		Since $X$ is affine, it corresponds to a single coloured cone $\sigma^c$ with $\calF(\sigma^c)=\calC=S\setminus I$. Using \cref{torus factor in cox construction}, we may assume that $X$ has no torus factors. 
		
		(1)$\Rightarrow$(2): Note that $\Gamma_{I\cup\calF(\sigma^c)}=\Gamma_S$ is the full Dynkin diagram of $G$. Therefore, the connected components of the Dynkin diagram are $\Gamma_{\{\alpha\}\cup I_\alpha}$ (using the notation of \cref{vivid definition}) as $\alpha$ ranges through the elements of $\calF(\sigma^c)=\calC$, and each connected component contributes a factor of a projective space in $G/P$ (using \cref{vivid Dynkin diagrams}).
		
		(2)$\Rightarrow$(3): This follows from \cite[Proposition 3.19]{monahan2023overview} because $X$ is affine.
		
		(3)$\Rightarrow$(4): It is well-known that $\Cox(X)$ is a polynomial ring when $X$ is a toric variety. 
		
		(4)$\Rightarrow$(1): This follows from combining \cref{vivid in cox construction} with the fact that affine space (being smooth) is vivid.
	\end{proof}
\end{proposition}

\begin{remark}[Geometry of vivid]\label{geometry of vivid}
	A $G/H$-horospherical variety $X$ is vivid if and only if each of the affine horospherical varieties in the affine local structure for $X$ (as in \cref{subsec:affine local structure}) is vivid. Combining with \cref{vivid for affine}, this tells us that vividness for $X$ is precisely indicating when the affine local structure of $X$ uses toric varieties. 
\end{remark}

\section{Main results}\label{sec:main results}

The goal of this section is to prove the main results from \cref{sec:introduction}, namely \cref{characterization of quotient singularities} and \cref{fulton's question for horospherical varieties}. For convenience, we recall the statement of \cref{characterization of quotient singularities}. 

\characterizationofquotientsingularities*

\begin{remark}[Quotient singularities over $k$ vs. $\C$]\label{quotient singularities over k vs C}
	Since $k$ is algebraically closed with characteristic $0$, we may base change a variety $X$ over $k$ to $X_\C$ over $\C$, and then by the Lefschetz principle, $X$ has quotient singularities if and only if $X_\C$ has quotient singularities. 
	
	Indeed, if $V/F\to X$ is an étale neighbourhood where $V$ is smooth and $F$ is a finite group, all defined over $k$, then this is all defined over the algebraic closure of a finitely generated field $k_0$ which is a subfield of $\C$. Now we may descend to $V_{k_0}/F_{k_0}\to X_{k_0}$ which is still étale, and then base change to $V_{\C}/F_{\C}\to X_\C$ which is again étale. We may also reverse this process: starting with everything defined over $\C$, descend to $k_0$, and then base change to $k$. 
\end{remark}

\begin{proof}[Proof of \cref*{characterization of quotient singularities}]
	In this proof, we may assume that $X$ has no torus factors; this will simplify our use of the Cox construction (see \cref{torus factor in cox construction}). 
	
	First, we prove the reverse direction, so suppose that $\Sigma^c$ is simplicial and vivid. Since quotient singularities are a local condition, we may assume that $\Sigma^c=\sigma^c$ is a single coloured cone, and using the affine local structure for horospherical varieties, we may assume that $X$ is affine, i.e. $\calF(\sigma^c)=\calC$; note that this simplification preserves the assumptions on $\Sigma^c$. In this case, $X$ is toric by \cref{vivid for affine}, so because it is $\Q$-factorial (since $\Sigma^c=\sigma^c$ is simplicial) it follows that $X$ has quotient singularities (see \cite[Theorem 3.1.19]{cox2011toric} and \cite[Proposition 4.2.7]{cox2011toric}). 
	
	Now we prove the forward direction by contraposition. If $\Sigma^c$ is not simplicial, then $X$ is not $\Q$-factorial, so $X$ must have worse than quotient singularities. Therefore, we may assume that $\Sigma^c$ is simplicial but not vivid. We begin by making several reductions. As in the proof of the reverse direction, we may assume that $X$ is affine, i.e. $\Sigma^c=\sigma^c$ and $\calF(\sigma^c)=\calC$. 
	
	Since $X$ is $\Q$-factorial, we can write $X=\wh X/\wh K$ (geometric quotient) where $\wh X$ and $\wh K$ come from the Cox construction for $X$. Note that $\wh X=\Spec(\Cox(X))$ since $X$ is affine. By \cite[Theorem 3.8]{gagliardi2014cox-ring}, we have $\Cox(X)=\Cox(G/P)[x_1,\ldots,x_\ell]$ where $\ell$ is the number of $G$-invariant prime divisors of $X$. Thus, $\wh X=Y\times\A^\ell$ where $Y:=\Spec(\Cox(G/P))$. 
	To show that $X$ has worse than quotient singularities it suffices to show this for $Y$. Recall from the Cox construction that $Y$ is an affine horospherical $G'$-variety (for some particular $G'$) which is factorial but not vivid, and its coloured cone $\sigma^c(Y)$ is generated by all colour points (which form a $\Z$-basis for the lattice) and $\calF(Y)=\calC$. 
	
	Now we explain how we can assume that $Y$ has a unique (therefore isolated) singularity. Consider a coloured face $\sigma_0^c=(\sigma_0,\calF_0)$ of $\sigma^c(Y)$ defined as follows: $\sigma_0$ is generated by all colour points of $\calF_0$, where $\calF_0$ is chosen minimally among subsets of $\calF(Y)=\calC$ such that $\sigma_0^c$ is not vivid (which exists because $Y$ is not vivid). Then the horospherical $G'$-variety $Y_0$ corresponding to $\sigma_0^c$ is a $G'$-invariant open subvariety of $Y$, so if we can show that $Y_0$ has worse than quotient singularities, then the same will be true for $Y$. 
	Again using the affine local structure for horospherical varieties, the singularities of $Y_0$ are the same as those of $Z_0$ where $Z_0$ is the corresponding affine horospherical variety (as in \cref{subsec:affine local structure}). 
	Note that, after removing all torus factors from $Z_0$, this $Z_0$ is the spectrum of the Cox ring of a flag variety, just like $Y$.
	By minimality of $\calF(Y_0)=\calF(Z_0)$, the fixed point in $Z_0$ is singular, but every other point is contained in a horospherical open subvariety which is vivid and thus smooth (since it is factorial). Therefore, in total we may assume that $Y$ (after replacing it with $Z_0$) has an isolated singularity $p\in Y$.  
	
	Let $G'/H'$ be the open $G$-orbit in $Y$ for an appropriate horospherical subgroup $H'$ (coming from the Cox construction); in fact, $G'/H'\cong G/[P,P]$. Since $Y$ has no torus factors, $H'$ contains the radical $\Rad(G')$ of $G'$, and therefore $G'/H'=(G'/\Rad(G'))/(H'/\Rad(G'))$. This tells us that we may assume $G'=(G')^{ss}$ is semisimple simply connected, and thus the same is true of $G$. Note that, because $Y$ has no $G'$-invariant divisors (equivalently, it has no non-coloured rays in its coloured fan), the codimension of $Y\setminus (G'/H')$ is at least $2$ in $Y$. 
	
	Using \cref{quotient singularities over k vs C}, we may assume that $k=\C$. If $p$ was a quotient singularity, then we could find an analytic neighbourhood (since $k=\C$, we can use analytic neighbourhoods instead of étale neighbourhoods) $W\subseteq Y$ of $p$ where $W=V/F$ with $V$ smooth and $F$ a finite group. Since $k=\C$, we may assume $V=\A^n$ and $F$ acts linearly. Then $F$ acts freely away from the origin in $V$ because $p\in W$ is an isolated singularity. Now $(V\setminus \{0\})\to (W\setminus\{p\})$ is a quotient map where $F$ acts freely, so it is a topological covering. Thus, if we can show that $\pi_1(W\setminus\{p\})$ is trivial, then $F$ must be trivial, which is a contradiction. 
	
	Notice that $W\setminus\{p\}$ is homotopy equivalent to $Y\setminus\{p\}$ (both $W$ and $Y$ are contractible since the former is affine space modulo a finite group, and the latter is the affine multicone over the flag variety $G/P$), so it suffices to show that $\pi_1(Y\setminus\{p\})$ is trivial. Since $Y\setminus\{p\}$ is smooth and $(Y\setminus\{p\})\setminus (G'/H')$ is a closed subvariety of codimension at least $2$ in $Y\setminus\{p\}$, it follows that $\pi_1(Y\setminus\{p\})=\pi_1(G'/H')=\pi_1(G/[P,P])$. Finally, $G$ is simply connected and $[P,P]$ is connected, so $\pi_1(G/[P,P])$ is trivial.
\end{proof}

\begin{example}\label{example Q-factorial not quotient singularity}
	The following is an explicit example of a horospherical variety which is $\Q$-factorial but has worse than quotient singularities, which could not happen in the world of toric varieties. Let 
	\[
	X:=\{\vec u\cdot \vec x=0\}=\{ux+vy+wz=0\}\subseteq\C^3_{\vec u}\times\C^3_{\vec x}=\C^6
	\]
	(where $\vec u=(u,v,w)$ and $\vec x=(x,y,z)$), which is an affine horospherical $\SL_3$-variety via the action $A\cdot (\vec u,\vec x):=(A\vec u,(A^{-1})^T\vec x)$. As in \cite[Example 5.10]{monahan2023overview}, the coloured fan for $X$ is regular but not vivid, so $X$ is factorial (hence $\Q$-factorial) but the singularity $0\in X$ is not a quotient singularity.
\end{example}

Using \cref{characterization of quotient singularities}, we now prove \cref{fulton's question for horospherical varieties} as stated below. 

\begin{corollary}\label{fulton's question for horospherical as corollary}
	\cref{fulton's question for horospherical varieties} holds. That is, if $X$ is a quasiprojective horospherical variety with quotient singularities, then $X=V/F$ for some smooth variety $V$ and some finite diagonalizable group $F$. 
	
	\begin{proof}
		We may assume that $X$ has no torus factors. By \cref{characterization of quotient singularities}, $X$ is vivid, so the Cox construction yields $X=\wh X/\wh K$ where $\wh X$ is smooth (using \cref{vivid in cox construction}) and $\wh K$ is diagonalizable acting on $\wh X$ with finite stabilizers (since $X$ is $\Q$-factorial; see \cite[Corollary 1.3.4.8]{arzhantsev2015cox-rings}). Therefore, we may apply \cite[Theorem 2.10]{satriano2015quotients} to obtain the desired result. 
	\end{proof}
\end{corollary}

\begin{remark}\label{fulton's global quotient for affine is horospherical}
	If $X$ is an affine horospherical variety with quotient singularities, then we can do \cref{fulton's question for horospherical varieties} in a ``horospherical way", meaning $V$ is a horospherical variety and the quotient map $V\to X$ is a horospherical morphism; compare this with \cref{fulton's global quotient need not be horospherical}. Indeed, in this case $X=\wh X/\wh K$ where $\wh X$ is smooth and $\wh K$ is finite. 
\end{remark}

\section*{Acknowledgements}

I am very thankful to Matthew Satriano for first introducing me to the question by Fulton and for having several helpful discussions during my time working on this paper. I would also like to thank Robert Harris, Brett Nasserden, Christian Liedtke, Matthias Pfeifer, and Stephen Melinyshyn for some helpful conversations, and the Mathematics Stack Exchange user \textit{Sasha} (see the post \href{https://math.stackexchange.com/questions/4782639/is-0-in-uxvywz-0-subseteq-mathbb-c6-a-quotient-singularity}{\textit{Is $0\in\{ux+vy+wz=0\}\subseteq\C^6$ a quotient singularity?}} from October 7, 2023) for inspiring a key step in the proof of \cref{characterization of quotient singularities}. Lastly, I am very grateful to the two anonymous referees for their detailed feedback.

\printbibliography 

@article {monahan2025horospherical-stack,
    AUTHOR = {Monahan, Sean},
     TITLE = {Horospherical stacks and stacky coloured fans},
   JOURNAL = {Transactions of the American Mathematical Society},
    VOLUME = {378},
      YEAR = {2025},
    NUMBER = {2},
     PAGES = {1167--1214},
      ISSN = {0002-9947,1088-6850},
   MRCLASS = {14M27 (05E14 14A20 14M15 14M17)},
  MRNUMBER = {4850437},
       DOI = {10.1090/tran/9297}
}

@misc{monahan2023overview,
	title={An overview of horospherical varieties and coloured fans}, 
	author={Sean Monahan},
	year={2023},
	eprint={2305.13558},
	archivePrefix={arXiv},
	primaryClass={math.AG}
}

@article{gagliardi2014cox-ring,
	AUTHOR = {Gagliardi, Giuliano},
	TITLE = {The {C}ox ring of a spherical embedding},
	JOURNAL = {Journal of Algebra},
	VOLUME = {397},
	YEAR = {2014},
	PAGES = {548--569},
	ISSN = {0021-8693},
	MRCLASS = {14M27},
	MRNUMBER = {3119238},
	DOI = {10.1016/j.jalgebra.2013.08.037}
}

@inproceedings{knop1991luna,
	AUTHOR = {Knop, Friedrich},
	TITLE = {The {L}una-{V}ust theory of spherical embeddings},
	BOOKTITLE = {Proceedings of the {H}yderabad {C}onference on {A}lgebraic
	{G}roups ({H}yderabad, 1989)},
	PAGES = {225--249},
	PUBLISHER = {Manoj Prakashan, Madras},
	YEAR = {1991},
	MRCLASS = {14M17 (14L30)},
	MRNUMBER = {1131314}
}

@article{perrin2018sanya,
	AUTHOR = {Perrin, Nicolas},
	TITLE = {Sanya lectures: geometry of spherical varieties},
	JOURNAL = {Acta Mathematica Sinica (English Series)},
	VOLUME = {34},
	YEAR = {2018},
	NUMBER = {3},
	PAGES = {371--416},
	ISSN = {1439-8516},
	MRCLASS = {14M27 (14-02 14L30)},
	MRNUMBER = {3763970},
	DOI = {10.1007/s10114-017-7163-6}
}

@article{luna1983plongements,
	AUTHOR = {Luna, D. and Vust, Th.},
	TITLE = {Plongements d'espaces homog\`{e}nes},
	JOURNAL = {Commentarii Mathematici Helvetici},
	VOLUME = {58},
	YEAR = {1983},
	NUMBER = {2},
	PAGES = {186--245},
	ISSN = {0010-2571},
	MRCLASS = {14M17 (14E25 14L30 20G15)},
	MRNUMBER = {705534},
	DOI = {10.1007/BF02564633}
}

@article{vinberg1972class-quasihomogeneous,
 	title={On a class of quasihomogeneous affine varieties},
 	author={Vinberg, {\`E} B and Popov, Vladimir Leonidovich},
 	journal={Mathematics of the USSR-Izvestiya},
 	volume={6},
 	number={4},
 	pages={743},
 	year={1972},
	MRCLASS = {14L10},
	MRNUMBER = {0313260}
}

@article{pauer1983glatte-einbettungen,
	AUTHOR = {Pauer, Franz},
	TITLE = {Glatte {E}inbettungen von {$G/U$}},
	JOURNAL = {Mathematische Annalen},
	VOLUME = {262},
	YEAR = {1983},
	NUMBER = {3},
	PAGES = {421--429},
	ISSN = {0025-5831,1432-1807},
	MRCLASS = {14L30},
	MRNUMBER = {692866},
	DOI = {10.1007/BF01456019}
}

@article{brion2007total-coordinate,
	AUTHOR = {Brion, Michel},
	TITLE = {The total coordinate ring of a wonderful variety},
	JOURNAL = {Journal of Algebra},
	VOLUME = {313},
	YEAR = {2007},
	NUMBER = {1},
	PAGES = {61--99},
	ISSN = {0021-8693,1090-266X},
	MRCLASS = {14L30 (13A50)},
	MRNUMBER = {2326138},
	DOI = {10.1016/j.jalgebra.2006.12.022}
}

@article{pasquier2017survey-singularities,
	AUTHOR = {Pasquier, Boris},
	TITLE = {A survey on the singularities of spherical varieties},
	JOURNAL = {EMS Surveys in Mathematical Sciences},
	VOLUME = {4},
	YEAR = {2017},
	NUMBER = {1},
	PAGES = {1--19},
	ISSN = {2308-2151,2308-216X},
	MRCLASS = {14M27 (14E15 14E30)},
	MRNUMBER = {3656466},
	DOI = {10.4171/EMSS/4-1-1}
}

@phdthesis{pasquier2006thesis,
	TITLE = {{Vari{\'e}t{\'e}s horosph{\'e}riques de Fano}},
	AUTHOR = {Pasquier, Boris},
	URL = {https://theses.hal.science/tel-00116977},
	SCHOOL = {{Universit{\'e} Joseph-Fourier - Grenoble I}},
	YEAR = {2006},
	TYPE = {Theses},
	HAL_ID = {tel-00116977},
	HAL_VERSION = {v1}
}

@article{gagliardi2015smoothness-spherical,
	AUTHOR = {Gagliardi, Giuliano},
	TITLE = {A combinatorial smoothness criterion for spherical varieties},
	JOURNAL = {Manuscripta Mathematica},
	VOLUME = {146},
	YEAR = {2015},
	NUMBER = {3-4},
	PAGES = {445--461},
	ISSN = {0025-2611,1432-1785},
	MRCLASS = {14M27 (14L30 20G05)},
	MRNUMBER = {3312454},
	DOI = {10.1007/s00229-014-0713-7}
}

@article{satriano2015quotients,
	AUTHOR = {Geraschenko, Anton and Satriano, Matthew},
	TITLE = {Torus quotients as global quotients by finite groups},
	JOURNAL = {Journal of the London Mathematical Society. Second Series},
	VOLUME = {92},
	YEAR = {2015},
	NUMBER = {3},
	PAGES = {736--759},
	ISSN = {0024-6107,1469-7750},
	MRCLASS = {14L30 (14D23 14M25)},
	MRNUMBER = {3431660},
	DOI = {10.1112/jlms/jdv046}
}

@book{cox2011toric,
	AUTHOR = {Cox, David A. and Little, John B. and Schenck, Henry K.},
	TITLE = {Toric varieties},
	SERIES = {Graduate Studies in Mathematics},
	VOLUME = {124},
	PUBLISHER = {American Mathematical Society, Providence, RI},
	YEAR = {2011},
	PAGES = {xxiv+841},
	ISBN = {978-0-8218-4819-7},
	MRCLASS = {14M25 (05A15 05E45 52B12)},
	MRNUMBER = {2810322},
	DOI = {10.1090/gsm/124}
}

@article{cox1995homogeneous,
	AUTHOR = {Cox, David A.},
	TITLE = {The homogeneous coordinate ring of a toric variety},
	JOURNAL = {Journal of Algebraic Geometry},
	VOLUME = {4},
	YEAR = {1995},
	NUMBER = {1},
	PAGES = {17--50},
	ISSN = {1056-3911},
	MRCLASS = {14M25},
	MRNUMBER = {1299003}
}

@book{arzhantsev2015cox-rings,
	AUTHOR = {Arzhantsev, Ivan and Derenthal, Ulrich and Hausen, J\"{u}rgen and
	Laface, Antonio},
	TITLE = {Cox rings},
	SERIES = {Cambridge Studies in Advanced Mathematics},
	VOLUME = {144},
	PUBLISHER = {Cambridge University Press, Cambridge},
	YEAR = {2015},
	PAGES = {viii+530},
	ISBN = {978-1-107-02462-5},
	MRCLASS = {14Cxx (14Jxx 14Lxx)},
	MRNUMBER = {3307753}
}

@book{timashev2011homogeneous,
	AUTHOR = {Timashev, Dmitry A.},
	TITLE = {Homogeneous spaces and equivariant embeddings},
	SERIES = {Encyclopaedia of Mathematical Sciences},
	VOLUME = {138},
	NOTE = {Invariant Theory and Algebraic Transformation Groups, 8},
	PUBLISHER = {Springer, Heidelberg},
	YEAR = {2011},
	PAGES = {xxii+253},
	ISBN = {978-3-642-18398-0},
	MRCLASS = {14M17 (14L30 14M27)},
	MRNUMBER = {2797018},
	DOI = {10.1007/978-3-642-18399-7}
}

@article{hu2000mori,
	AUTHOR = {Hu, Yi and Keel, Sean},
	TITLE = {Mori dream spaces and {GIT}},
	JOURNAL = {Michigan Mathematical Journal},
	VOLUME = {48},
	YEAR = {2000},
	PAGES = {331--348},
	ISSN = {0026-2285,1945-2365},
	MRCLASS = {14L24 (14E30)},
	MRNUMBER = {1786494},
	DOI = {10.1307/mmj/1030132722}
}
	
\end{document}